\documentclass{article}
\usepackage[utf8]{inputenc}

\usepackage{authblk}
\usepackage[font=small, format=plain, labelfont={bf,it}, textfont=it]{caption}
\usepackage{fullpage}
\usepackage{amssymb, amsmath, tikz, stmaryrd, amsfonts, latexsym, amscd, amsthm, enumerate, enumitem, epstopdf, graphicx, caption, float, multirow, url, epstopdf, xcolor, enumerate, fancyhdr, afterpage}

\usepackage[all]{xy}
\usepackage[backref]{hyperref}
\usepackage{url,hyperref}
\usepackage{tikz}
\usetikzlibrary{arrows}
\usepackage{verbatim, color}
 \usepackage{mathtools}

\DeclarePairedDelimiter\floor{\lfloor}{\rfloor}

\renewcommand{\Re}{\mathrm{Re}}
\renewcommand{\Im}{\mathrm{Im}}

\newtheorem*{theorem*}{Theorem}
\newtheorem*{conjecture*}{Conjecture}
\newtheorem*{corollary*}{Corollary}
\newtheorem*{teo*}{Theorem}
\newtheorem*{conj*}{Conjecture}
\newtheorem{theorem}{Theorem}[section]
\newtheorem{lemma}[theorem]{Lemma}
\newtheorem{proposition}[theorem]{Proposition}
\newtheorem{prop}[theorem]{Proposition}
\newtheorem{corollary}[theorem]{Corollary}

\theoremstyle{definition}
\newtheorem{definition}[theorem]{Definition}
\newtheorem{example}[theorem]{Example}

\newtheorem{remark}[theorem]{Remark}

\newcommand{\Z}{\mathbb{Z}}

\newcommand{\Q}{\mathbb{Q}}

\DeclareMathOperator{\Gal}{Gal}

\newif\ifcomment

\commentfalse 

\newcommand{\cami}[1]{\ifcomment \textcolor{blue}{Cami: #1} \fi}
\newcommand{\laia}[1]{\ifcomment \textcolor{orange}{Laia: #1} \fi}
\newcommand{\tao}[1]{\ifcomment \textcolor{green!50!black}{Tao: #1} \fi}

\newif\ifcommentLater
\commentLaterfalse

\usepackage{etoolbox}
\makeatletter
\patchcmd{\maketitle}
 {\def\@makefnmark}
 {\def\@makefnmark{}\def\useless@macro}
 {}{}
\makeatother

\begin{document}
\title{Even Unimodular Lattices from Quaternion Algebras\footnotetext{\llap{\textsuperscript{}}At the time of carrying out most of this work, L.~Amor\'{o}s was with the Helsinki Institute of Information Technology, Aalto University, and is now with the Finnish Meteorological Institute; M. T. Damir was with the Department of Mathematics and Systems Analysis,  Aalto University, and is currently with the Department of Computer Science, University of Helsinki. C.~Hollanti is with the Department of Mathematics, Aalto University. 

M.~T.~Damir would like to thank Aalto Science Institute for partial funding towards his research. C.~Hollanti's work was supported in part by the Academy of Finland grant \#336005, which is gratefully acknowledged. \\
Emails: laia.amoros@fmi.fi,
mohamed.damir@helsinki.fi, 
camilla.hollanti@aalto.fi}}
\author{Laia Amor\'{o}s$^{1,2}$, M. Taoufiq Damir$^{1,3}$, Camilla Hollanti$^1$}
\date{}

\affil{
$^1$Aalto University, $^2$Finnish Meteorological Institute, $^3$University of Helsinki 
}
\maketitle

\cami{Unify notation: D or d for $\mathbb{Q}(\sqrt{D})$? Somewhere it also reads $d_F=D$, which is now a bad choice wrt to the former..}
\laia{We should have: $D$ for $\mathbb{Q}(\sqrt{D})$, and $d_F$ the discriminant, which might or might not be $D$.}

\begin{abstract}
    We review a lattice construction arising from quaternion algebras over number fields and use it to obtain some known extremal and densest lattices in dimensions 8 and 16. The benefit of using quaternion algebras over number fields is that the dimensionality of the construction problem is reduced by 3/4. We explicitly construct the $E_8$ lattice (resp. $E_8^2$ and $\Lambda_{16}$) from infinitely many quaternion algebras over real quadratic (resp. quartic) number fields and we further present a density result on such number fields. \cami{only for dim 8 now? In 16 just say that it's an infinite set.} By relaxing the extremality condition, we also provide a source for constructing even unimodular lattices in any dimension multiple of $8$. 
\end{abstract}

\section{Introduction}

Lattices are objects that naturally arise in many areas of mathematics and  have been used in a wide range of applications, including communications and cryptography. 
A \textit{lattice} $\Lambda$ is a discrete additive subgroup of $\mathbb{R}^n$, generated by all integral combinations of $k$ linearly independent vectors $\{v_1,\ldots,v_k\}$ of $\mathbb{R}^n$. If $k=n$ we say that $\Lambda$ is a \textit{full-rank} lattice.
In this paper we are interested in lattices arising from number theoretic constructions. 
Lattices arising from ideals in number fields, known as \textit{ideal lattices}, have been extensively studied \cite{EB1, EB2, EB3}. 
A natural extension of ideal lattices is the construction of lattices arising from orders in division algebras, and in particular in quaternion algebras. \cami{add reference(s)?}

A lattice $\Lambda$ is said to be \textit{integral} if the squared norms of its vectors are integers. If these norms are even, then the lattice is called \textit{even}. 
An integral lattice of volume one is said to be \textit{unimodular}. A key property in the study of unimodular lattices is based on the fact that their associated theta series are modular forms. This played an important role in deriving results about such lattices using the analytic theory of modular forms.
A well known result in this direction shows that even unimodular lattices exist only in dimensions multiple of $8$. Moreover, a celebrated theorem of Siegel shows that the length of a shortest vector in an even unimodular lattice of dimension $n$ is at most $\floor{\frac{n}{24}}+1$. Lattices achieving this bound are called \textit{extremal}.

Classifying extremal even unimodular lattices in every dimension is a long standing problem in the theory of lattices. Unfortunately, this classification is only known up to dimension $24$, namely, in dimensions $8, 16$ and $24$. It is worth mentioning that the optimal lattice sphere packings in dimensions $8$ and $24$ are extremal even unimodular. Currently, the exact number of such lattices is still an open problem in higher dimensions. For example, we only know of a single one in dimension $72$ \cite{nebe2016automorphisms} and four in dimension $80$.

In this paper we consider lattices arising from maximal orders in quaternion algebras over real number fields. Under some verifiable conditions (see Section \ref{sec:extremal}), we obtain even (integral) lattices of dimension $4n$, where $n$ is the degree of the underlying number field. 
A related study appears in the work of Bayer and Martinet \cite{semisimple}, where the authors study quadratic forms related to semi-simple algebras. We extend this study by giving explicit constructions of the extremal even unimodular lattices in dimensions $8$ and $16$. Moreover we give a counting estimate on the number of number fields giving rise to such lattices. 
Namely, we show that a lattice constructed from a maximal order in a quaternion algebra over a real number field is even unimodular with a high probability. 
Unfortunately, this result cannot be extended to dimensions $24,32$ and $40$ (see \cite{semisimple}), but by relaxing the extremality condition, we can still provide a source for constructing even unimodular lattices in any dimension multiple of $8$.

\subsection{Related work and contributions}
We consider lattices arising from quaternion algebras over totally real number fields. This is the same setting as the one taken in \cite{Tu-Yang} and \cite{Costa2020} via a twisted embedding of quaternion algebras.
In \cite{Tu-Yang} the authors use the symmetric bilinear form associated to the trace form of the quaternion algebra to obtain a formula for the determinant of the associated lattice. Then they use this formula to give some examples for constructions of the lattices $E_8$, $K_{12}$, $\Lambda_{16}$, $\Lambda_{24}$ and a densest 32-dimensional lattice. This construction is also the one considered in \cite{semisimple}, where the authors show that many well-known lattices can be constructed using this setting. Their method, however, does not provide a constructive way of obtaining a generator matrix of the lattices obtained.

In \cite{Costa2020} the authors describe how to construct $4n$-dimensional lattices from ideals in maximal orders from quaternion algebras over totally real number fields of degree $n$.
They use this to construct a specific family of lattices arising from the quaternion algebras $B=\left(\frac{-1,-1}{F}\right)$, where $F$ denotes a totally real maximal subfield of $\mathbb{Q}(\zeta_{2^r})$, of the form $F=\mathbb{Q}(\zeta_{2^r}+\zeta_{2^r}^{-1})$ , where $r\geq2$ and $\zeta_{2^r}$ denotes a $2^r$-th root of unity. 
The lattices they obtain include the densest packings known in dimension 1,2,4,8 and 16.
Once one has a basis for the maximal order,
their construction allows to explicitly compute a generator matrix and the minimum distance of the lattice.
One of the main difficulties in their method is to find a maximal order for a fixed quaternion algebra. There is no known efficient algorithm for finding maximal orders in quaternion algebras over number fields other than $\mathbb{Q}$.
In \cite{Costa2020} the authors find an infinite family of maximal orders in the quaternion algebras mentioned above. More precisely, for each quaternion algebra $B=\left(\frac{-1,-1}{F}\right)$, with $F=\mathbb{Q}(\zeta_{2^r}+\zeta_{2^r}^{-1})$, they give a basis for a maximal order in $B$. Here, we extend this work by providing infinitely many explicit maximal orders \emph{per dimension} in dimensions 8 and 16, hence obtaining a huge variety of candidate lattices all providing a dense lattice packing, as explained below in more detail.

In this paper we use the same underlying lattice construction as in \cite{Costa2020}, and we use the associated lattice volume formula to ensure that our lattices are unimodular. We consider quaternion algebras of the form $B=\left(\frac{-1,-1}{F}\right)$, where $F$ is a totally real number field of degree $n>1$.
Then for $n=2,4$ we construct an infinite family of totally real number fields $F$ of degree $n$ for which we control the discriminant, and we find a totally real element $\alpha\in F$ in such a way that, for any maximal order $\mathcal{O}$ in $B$, the associated lattice $\Lambda_{(\mathcal{O},\alpha)}$ is an extremal even unimodular lattice or, in some cases, a densest lattice packing. That is, for $n=2$ we find infinitely many quaternion algebras from which we can construct the lattice $E_8$, and for $n=4$ we find infinitely many quaternion algebras from which we obtain either $\Lambda_{16}$ (densest) or $E_8^2$ (extremal).
Our method can be easily extended to different values of $n$. It is interesting to notice that different maximal orders in the quaternion algebras we consider give different even unimodular lattices, whose extremality can be checked and recognised by computing their center density. The constructions are explicit as we are also able to characterise a maximal order basis for the quaternion algebras we consider.

The known constructions of even unimodular lattices in a given dimension $n$ are often based on algebraic structures of the same dimension (number fields, codes). In the present paper, we first benefit from the structure of quaternion algebras over number fields by reducing the dimensionality of the problem by $3/4$. In other words, we will use a number field of degree $n$ to construct even unimodular lattices of dimension $4n$.
Our aim is to further illustrate that such constructions also provide a significant freedom on the choice of the underlying number field, and we present  density results for the obtained sets of even unimodular lattices.

\subsection{Organization}
The paper is organised as follows.
In Section \ref{sec:background} we review the necessary background on lattices arising from number fields and we introduce the reader to quaternion algebras.
In Section \ref{sec:quaternions} we study in more detail definite quaternion algebras of the form $\left(\frac{-1,-1}{F}\right)$, where $F$ denotes a totally real number field. We characterise bases of maximal orders over real quadratic fields and over a family of quartic fields called the \textit{simplest quartic fields}, that will allow us to construct infinite families of lattices.
In Section \ref{1st} we describe the lattice construction from quaternion algebras over totally real number fields and we introduce the necessary notation that we will use afterwards.
In Section \ref{sec:extremal} we turn to extremal even unimodular/dense lattices. We give constructive results that provide an infinite family of lattices similar to $E_8$ and an infinite family of lattices similar to $\Lambda_{16}$ obtained from infinitely many quaternion algebras. We also provide a density result showing that a large proportion of the quaternion algebras we consider give rise to these lattices.
In Section \ref{sec:apps} we explain how the lattices we construct can have potential interest in applications such as code design for wireless communications and physical layer security and hint to possible future work.
Finally we include some Magma code to generate our lattices in the Appendix.

\subsection*{Notation}

    Throughout the paper, $F$ will denote a number field, $\mathbb{Z}_F$ will denote the ring of integers of $F$ and $d_F$ the discriminant of $F$.
    The norm of an element $x$ in $F$ is denoted by $\mathrm{Nm}_{F/\mathbb{Q}}(x)$ and the trace of $x$ as $\mathrm{Tr}_{F/\mathbb{Q}}(x)$.
    We denote a quaternion algebra by $B$, and $\mathcal{O}$ denotes an order in $B$.

\section{Background}\label{sec:background}
    
    In this section we will introduce the necessary notions of quaternion algebras and lattice theory that we will need in what follows.

\subsection{Lattices from number fields}\label{sec:background_lattices}
    
    Let $\Lambda$ be a full-rank lattice in $\mathbb{R}^n$, with $n\geq2$. The \textit{minimum} of $\Lambda$ is defined as
	$$\lambda_1(\Lambda):=\min\{||x||^2 : x\in\Lambda\setminus\{0\}\}.$$
	The \textit{centre density} of $\Lambda$ is defined as
	$$\delta_\Lambda:=\frac{||\lambda_1(\Lambda)||^{n/2}}{2^n\sqrt{\det\Lambda}}.$$

	


	
	Let $F$ be a number field of degree $n$ over $\mathbb{Q}$.
    Let $r$ denote the number of real embedding $\{\sigma_1,\ldots,\sigma_r\}$ and $2s$ the number of complex embeddings $\{\sigma_{r+1}, \overline{\sigma_{r+1}},\ldots,\sigma_{r+s},\overline{\sigma_{r+s}}\}$ of $F$.
    We have a canonical embedding
    $$\begin{array}{rrcl}
        \sigma_F: & F & \rightarrow & \mathbb{R}^r\times\mathbb{C}^s \\
        & x & \mapsto & (\sigma_1(x),\ldots,\sigma_{r+s}(x)).
    \end{array}$$
    We can identify $\mathbb{R}^r\times\mathbb{C}^s$ with $\mathbb{R}^n$, using the basis $\{1,i\}$ for $\mathbb{C}$.
    
    Let $\mathcal{I}$ be a nonzero ideal in $\mathbb{Z}_F$ with integral basis $\left\{\alpha_1,\ldots,\alpha_n\right\}$. Then $\sigma_F(\mathcal{I})$ is a lattice with generator matrix whose $i$-th row is
    $$(\sigma_1(\alpha_i),\ldots,\sigma_r(\alpha_i),\Re(\sigma_{r+1}(\alpha_i)), \Im(\sigma_{r+1}(\alpha_i)), \ldots, \Re(\sigma_{r+s}(\alpha_i)), \Im(\sigma_{r+s}(\alpha_i))).$$
    We call $\mathcal{I}$ an \textit{ideal lattice}.

\subsection{Arithmetic of quaternion algebras}

    \begin{definition}
        A \textit{quaternion algebra} over $F$ is a central simple algebra that has dimension $4$ over $F$. For $a,b\in\ F \setminus\{0\}$ we denote by $\left(\frac{a,b}{F}\right)$ the $F$-algebra generated by a basis $\{1,i,j,k\}$ such that  $i^2=a, j^2=b,$ and $ij=-ji=k.$
    \end{definition}
    
    Every quaternion algebra $B$ over $F$ is provided with an $F$-endomorphism called \textit{conjugation} and denoted by $\beta\mapsto\overline{\beta}$. If 
    $\beta=x+yi+zj+tk\in B$, with $x,y,z,t\in F$, then $\overline{\beta}=x-yi-zj-tk$. The \textit{reduced trace} of $\beta$ is defined as $\mathrm{Tr}_B(\beta)=\beta+\overline{\beta}=2x$ and the \textit{reduced norm} of $\beta$ is defined as $\mathrm{Nm}_B(\beta)=\beta\overline{\beta}=x^2-ay^2-bz^2+abt^2$.

    According to the Wedderburn--Artin theorem \cite{Vigneras1980}, a quaternion algebra is either a division algebra (a skew field) or a matrix algebra. We will only be interested in division algebras.
	For each place $\mathfrak{p}$ of $F$, $B_\mathfrak{p}:=F_\mathfrak{p}\otimes B$ is a quaternion algebra, where $F_\mathfrak{p}$ denotes the localization of $F$ at $\mathfrak{p}$.
	We say that $B$ is \textit{ramified} at $\mathfrak{p}$ if $B_\mathfrak{p}$ is a division algebra. Otherwise we say that $B$ is \textit{non-ramified} or \textit{split} at $\mathfrak{p}$. 
    As usual, the archimedean places play a distinguished role.
    Let $n$ denote the degree of $F$. The Wedderburn--Artin theorem tells us that we have an isomorphism
	$$B\otimes_\mathbb{Q}\mathbb{R}\simeq \mathrm{M}_2(\mathbb{R})^s\times\mathbb{H}^{n-s},$$
	for some $0\leq s\leq n$, where $\mathbb{H}=\left(\frac{-1,-1}{\mathbb{R}}\right)$ denotes the real Hamilton quaternion algebra.
	We say that $B$ is a (\textit{totally}) \textit{definite} quaternion algebra if $B$ is ramified at all the archimedean places, \emph{i.e.}, if $s=0$. Otherwise we say that $B$ is (\textit{totally}) \textit{indefinite}.

	\begin{definition}
		The \textit{reduced discriminant} $D_B$ of a quaternion algebra $B$ over $F$ is the integral ideal of $\mathbb{Z}_F$ given by the product of prime ideals of $\mathbb{Z}_F$ that ramify in $B$.
	\end{definition}
	
	Let $B$ be quaternion algebra over a number field $F$. An \emph{ideal} $\mathcal{I}$ of $B$ is a $\mathbb Z_F$-lattice of $B$ of rank $4$.
	The \textit{reduced norm} $\mathrm{Nm}_B(\mathcal{I})$ of an ideal $\mathcal I$ is $\text{gcd}\{\mathrm{Nm}_B(\beta) : \beta\in \mathcal I \}$.
	An \emph{order} $\mathcal{O}$ of $B$ is an ideal which is also a subring.
    A \emph{maximal order} is an order that is not properly contained in another order.
	
	The \textit{discriminant} $D_\mathcal{O}$
	of a quaternion order $\mathcal{O}=\mathbb{Z}_F[\alpha_1,\alpha_2,\alpha_3,\alpha_4]$ in $B$ can be computed as the ideal
    $D_\mathcal{O}^2=\det(\mathrm{Tr}_B(\alpha_i\alpha_j))_{1\leq i,j\leq 4}\mathbb{Z}_F$.
    Check Def. 1.31 in \cite{AlsinaBayer2004} for the precise definition.
	To check the maximality of $\mathcal{O}$ one can use the fact that an order $\mathcal{O}$ in $B$ is maximal if and only if $D_\mathcal{O}=D_B$ (see \cite{AlsinaBayer2004}, Prop. 1.50).
	Note that, unlike in number fields, maximal orders in quaternion algebras are not necessarily unique \cite{Vigneras1980}, but they all have the same discriminant, which coincides with the discriminant of the algebra.

\section{Computing discriminants and orders in quaternion algebras}\label{sec:quaternions}

    In order to explicitly construct lattices from quaternion algebras, we will need to be able to find bases of maximal orders explicitly. In this section we will describe some infinite families of maximal orders in a given dimension and construct an explicit  basis for each one.
    
\subsection{Computing the discriminant of definite quaternion algebras}
	
	We will study in detail the case of definite quaternion algebras of the form $B=\left(\frac{-1,-1}{F}\right)$, with $F$ a totally real number field, as this will allow us to do explicit computations in the sections that follow. First we need to compute the reduced discriminant of $B$.
    
    \begin{definition}
	    Let $L$ be a local field. The \textit{Hilbert symbol} over $L$ is defined as the function $ (\,\cdot\, ,\,\cdot\,)_L: L^\times \times L^\times \rightarrow \{\pm 1\}$ given by
	    $$(a,b)_{L} =\left\{\begin{array}{rl}
			1 & \mbox{if}\ ax^2+by^2=z^2\ \mbox{has a non-zero solution}\ (x,y,z)\ \mbox{in}\ L^3, \\
			    -1 & \mbox{otherwise}.
	   \end{array}\right.$$
	\end{definition}
	
     The Hilbert symbol is closely related to the ramification of a quaternion algebra $B=\left(\frac{a,b}{F}\right)$, with $a,b\in F^\times$. Indeed, $B$ is ramified at a place $\mathfrak{p}$ of $F$ if and only if $(a,b)_{F_\mathfrak{p}}=-1$ (cf. \cite{Vigneras1980}, Chapter III).
   
   	\begin{prop}\label{discB}
	    Let $F$ be a totally real Galois extension of $\mathbb{Q}$ of degree $n=[F:\mathbb{Q}]$ and let $\mathbb{Z}_F$ denote its ring of integers.
		Let $B=\left(\frac{-1,-1}{F}\right)$ be a quaternion algebra over $F$. Then all the archimedean places are ramified, so $B$ is a definite quaternion algebra, and
        \vspace{-2mm}
		\begin{itemize}\setlength\itemsep{0em}
			\item[$(a)$] if $n=2$ and F has discriminant $d_F$, then
			$$D_B=\left\{\begin{array}{lll}
			   2\mathbb{Z}_F & \emph{if} & d_F\equiv 1\quad (\emph{mod}\ 8), \\
			    \mathbb{Z}_F & \emph{if} & d_F\not\equiv 1\quad (\emph{mod}\ 8).
			   \end{array}\right.$$
			 
			 \item[$(b)$] if $n=2^k$, for some $k\in\mathbb{Z}$, then $D_B=\mathbb{Z}_F$.
			 
			 \item[$(c)$] if $n$ is odd, then $D_B=2\mathbb{Z}_F$.
			 
		    \item[$(d)$] if $F$ has degree $n=2^km$, where $m>1$ is odd, let $n_2=[F_\mathfrak{p}:\mathbb{Q}_2]$, where $\mathfrak{p}$ is a prime above $2$. Then 
		     $$D_B=\left\{\begin{array}{lll}
			    2\mathbb{Z}_F & \emph{if} & n_2\  \emph{odd}, \\
			    \mathbb{Z}_F & \emph{if} & n_2\ \emph{even}.
			 \end{array}\right.$$
		\end{itemize}
		
	\end{prop}
	
	\begin{proof}
        The ramification of the archimedean places can be easily computed with the Hilbert symbol. Indeed, the equation $-x^2-y^2=z^2$ has no solution in $F_\infty\simeq\mathbb{R}$, since $F$ is totally real, thus $(-1,-1)_{F_\infty}=-1$.
	  
	    For every non-archimedean place $\mathfrak{p}$ of $F$ we want to compute the Hilbert symbol $(-1,-1)_{F_{\mathfrak{p}}}$.
	    Let $p$ the prime integer below $\mathfrak{p}$, \emph{i.e.}, $\mathfrak{p}\mid p$. Consider the decomposition of $p$ in $\mathbb{Z}_F$: $p\mathbb{Z}_F=\mathfrak{p}_1^{e_1}\ldots\mathfrak{p}_g^{e_g}$, where $n=\sum_{i=1}^g e_if_i$, $f_i$ is the inertia degree of the $\mathfrak{p}_i$, and we set $\mathfrak{p}=\mathfrak{p}_1$.
	    We have that $[F_{\mathfrak{p}_i}:\mathbb{Q}_p]=e_if_i$ (cf. \cite{Janusz} Theorem 5.1).
	    
	    We will prove first the quadratic case. If $F=\mathbb{Q}(\sqrt{D})$ is a quadratic field, then $[F_{\mathfrak{p}}:\mathbb{Q}_p]=1$ if and only if $g=2$, which means that $p$ splits completely and $F_{\mathfrak{p}}=\mathbb{Q}_p$. In this case we can apply the well-known formulas of the Hilbert symbol (see \cite{Serre73}, Chapter III) and obtain:
	    $$(-1,-1)_{\mathbb{Q}_p}=\left\{ 
	    \begin{array}{rll}
	    1 & \mathrm{if} & p\neq2, \\
	    -1 & \mathrm{if} & p=2.\end{array}\right.$$
	    We only need to show when does 2 split completely in $F$. This happens if and only if $\left(\frac{D}{2}\right)=1$, and this is true if and only if $D\equiv 1$ (mod $8$). So we have shown so far that if $D\equiv 1$ (mod $8$), then the only primes that ramify are those above 2, and $D_B=2\mathbb{Z}_F$. Now if $D\not\equiv1$ (mod $8$), then $[F_{\mathfrak{p}_i}:\mathbb{Q}_p]=2$.
	    By \cite{Janusz} Chapter II, Theorem 5.1 , we have that the field $F_{\mathfrak{p}}$ is isomorphic to some finite extension $L=\mathbb{Q}_p(\beta)$ of $\mathbb{Q}_p$.
	    In this case, by \cite{Fesenko} Chapter 5, we have that
    	$$(-1,-1)_{F_\mathfrak{p}}=(-1,\mathrm{Nm}_{L/\mathbb{Q}_p}(-1))_{\mathbb{Q}_p}=(-1,1)_{\mathbb{Q}_p}=1.$$
    	So $D_B=\mathbb{Z}_F$.
    	
	    Let us next prove  the case where $F$ is a number field of degree $n>2$. Again, we denote by $L=\mathbb{Q}_p(\beta)$ the finite extension of $\mathbb{Q}_p$ isomorphic to $F_\mathfrak{p}$.
	    We have:
	    $$(-1,-1)_{F_\mathfrak{p}}= (-1,\mathrm{Nm}_{L/\mathbb{Q}_p}(-1))_{\mathbb{Q}_p}=\left\{ 
	    \begin{array}{ll}
	    (-1,1)_{\mathbb{Q}_p}=1 & \mathrm{if}\ [L:\mathbb{Q}_p]\  \mathrm{even}, \\
	    (-1,-1)_{\mathbb{Q}_p}=\left\{\begin{array}{rll}
	    1 & \mathrm{if} & p\neq2, \\
	    -1 & \mathrm{if} & p=2.\end{array}\right. & \mathrm{if}\  [L:\mathbb{Q}_p]\  \mathrm{odd}.\end{array}\right.$$
	    Notice that $[L:\mathbb{Q}_p]$ divides $[F:\mathbb{Q}]$. Fix $p=2$. If $n=2^k$, for some $k$, then $(-1)^{[L:\mathbb{Q}_2]}=(-1)^{[F:\mathbb{Q}]}=1$, and if $n$ is odd then
	    $(-1)^{[L:\mathbb{Q}_2]}=1$.
	\end{proof}

\subsection{Computing bases of maximal orders}

\subsubsection{Maximal orders over real quadratic fields}

    \begin{lemma}\label{d3}\cami{Erlier we had capital $D$ under the root, some reason why we change here?}\laia{I don't understand}\cami{I guess Tao already changed from lower-case d to capital D :). Tao, please leave a comment when you change things and comments become obsolete!}
		 Let $D>0$ be a square-free integer with $D\equiv3$ (mod $4$) and consider the quaternion algebra $B=\left(\frac{-1,-1}{\mathbb{Q}(\sqrt{D})}\right)$. Then $\mathcal{O}=\mathbb{Z}[\sqrt{D}][1,i,(\sqrt{D}i+j)/2, (\sqrt{D}+k)/2]$ is a maximal order in $B$.
	\end{lemma}
	
	\begin{proof}
		First note that $(\sqrt{D}i+j)/2, (\sqrt{D}+k)/2\in\mathcal{O}$, since $D\equiv 3$ (mod $4$).
		It is easy to see that $\mathcal{O}$ is an order, one only needs to check that for any $\alpha=x+yi+z(\sqrt{D}i+j)/2+t(\sqrt{D}+k)/2\in\mathcal{O}$, with $x,y,z,t\in\mathbb{Z}_F$ we have: 
		$(i)$ $\mathrm{Tr}_B(\alpha),\mathrm{Nm}_B(\alpha)\in \mathbb{Z}_F$; 
		$(ii)$ $\mathbb{Z}_F\subseteq\mathcal{O}$; and 
		$(iii)$ $\mathcal{O}$ contains an $F$-basis of $B$.
		
		To check the maximality of $\mathcal{O}$ we use the fact that an order $\mathcal{O}$ in $B$ is maximal if and only if $\mathrm{disc}(\mathcal{O})=D_B$ (see \cite{AlsinaBayer2004}, Prop. 1.50), where $\mathrm{disc}(\mathcal{O})$ denotes the \textit{reduced discriminant} of $\mathcal{O}$ (check Def. 1.31 in \cite{AlsinaBayer2004} for the precise definition). The reduced discriminant of an order $\mathcal{O}=\mathbb{Z}_F[\gamma_1,\gamma_2,\gamma_3,\gamma_4]$ can be computed as the ideal $\mathrm{disc}(\mathcal{O})^2=\det(\mathrm{Tr}(\gamma_i\gamma_j)_{i,j})\mathbb{Z}_F$. So in our case we have\cami{Check d vs D in the matrix!}
		$$\mathrm{disc}(\mathcal{O})^2=\det\left(\begin{smallmatrix} 2 & 0 & 0 & \sqrt{D} \\ 0 & -2 & -\sqrt{D} & 0 \\ 0 & -\sqrt{D} & -(D+1)/2 & 0 \\ \sqrt{D} & 0 & 0 & (D-1)/2 \end{smallmatrix}\right)\mathbb{Z}_F=-1\mathbb{Z}_F=\mathbb{Z}_F.$$
		Thus, $\mathcal{O}$ is indeed a maximal order in $B$.
	\end{proof}

	\begin{lemma}\label{d1}
		Let $D>0$ be a square-free integer with $D\equiv1$ (mod $4$) and consider the quaternion algebra $B=\left(\frac{-1,-1}{\mathbb{Q}(\sqrt{D})}\right)$.
		\begin{itemize}
			\item[$(i)$] If $D\not\equiv 1$ (mod $8$), then
			$\mathcal{O}=\mathbb{Z}_F\left[1,i,\frac{\sqrt{D}+1}{4}+\frac{\sqrt{D}+3}{4}i+\frac{j}{2},\frac{\sqrt{D}+3}{4}+\frac{\sqrt{D}+1}{4}i+\frac{k}{2}\right]$
			is a maximal order in $B$.
			
			\item[$(ii)$] If $D\equiv 1$ (mod $8$), then
			$\mathcal{O}=\mathbb{Z}_F[1,i,j,\frac{1+i+j+k}{2}]$
			is a maximal order in $B$.
		\end{itemize}		
	\end{lemma}
	
	\begin{proof}
	$(i)$ Suppose that $D\equiv 1$ (mod 4) and $D\not\equiv 1$ (mod $8$). Like in Lemma \ref{d3}, it is easy to check that $\mathcal{O}$ is an order. For the maximality we compute its discriminant, which again gives $\mathrm{disc}(\mathcal{O})=\mathbb{Z}_F$. By Proposition \ref{discB}, in this case $\mathrm{disc}(\mathcal{O})=D_B$, so $\mathcal{O}$ is maximal.
		
	\medskip
	$(ii)$ Now suppose that $D\equiv 1$ (mod $8$). Again, checking that $\mathcal{O}$ is an order is an easy task. In this case, the discriminant is
	$$\mathrm{disc}(\mathcal{O})^2=\det\left(\begin{smallmatrix} 2 & 0 & 0 & 1 \\ 0 & -2 & 0 & -1 \\ 0 & 0 & -2 & -1 \\ 1 & -1 & -1 & -1 \end{smallmatrix}\right)\mathbb{Z}_F=-4\mathbb{Z}_F,$$
	so $\mathrm{disc}(\mathcal{O})=2\mathbb{Z}_F=D_B$, according to Prop. \ref{discB}.
	\end{proof}
	
	\begin{lemma}\label{d2}
		Consider the quaternion algebra $B=\left(\frac{-1,-1}{\mathbb{Q}(\sqrt{2})}\right)$. Then
		$\mathcal{O}=\mathbb{Z}[\sqrt{2}][1,\frac{1+i}{\sqrt{2}},\frac{1+j}{\sqrt{2}},\frac{1+i+j+k}{2}]$ is a maximal order in $B$.
	\end{lemma}
	
	\begin{proof}
		\cite{Vigneras1980} p. 141.
	\end{proof}

\subsubsection{Maximal ordes over totally real quartic fields}

We consider the so called \textit{simplest quartic fields} (more details are given in Section \ref{subsec:dim16}).
For any integer $m\in\mathbb{N}\setminus\{3\}$ such that the odd part of $m^2 +16$ is square-free, we consider the polynomial
$$f_m(x)=x^4-mx^3-6x^2+mx+1$$
of discriminant $d_{f_m}=4\Delta_m^3$, where $\Delta_m=m^2+16$. 
The quartic number field $F_m$ defined by the polynomial $f_m$ is cyclic. The fields $F_m$ are called the simplest quartic fields \cite{stephane}. Let $r_m$ be the largest root of $f_m$. The root $r_m$ is explicitly given by \cite{stephane} as
    $$r_m=\frac{1}{2}\left( \frac{m+\sqrt{\Delta_m}}{2}+\frac{\sqrt{\Delta_m+m\sqrt{\Delta_m}}}{2}\right).$$
    In this case we can write $F_m=\mathbb{Q}(r_m)$.

\begin{lemma}\label{order_quartic}
    Let $F_m$ be the quartic field defined by $f_m(x)=x^4-mx^3-6x^2+mx+1\in\mathbb{Z}[x]$, with $m\in\mathbb{N}\setminus\{3\}$ such that $m$ is even and the odd part of $m^2 +16$ is square-free.
    Consider the quaternion algebra $B=\left(\frac{-1,-1}{F_m}\right)$. Then
    the discriminant of $B$ is $\mathbb{Z}_{F_m}$ and $\mathcal{O}=\mathbb{Z}_{F_m}\left[1, \frac{(1+r_m)(1+i)}{2}, \frac{(1+r_m)(1+j)}{2}, \frac{1+i+j+k}{2}\right]$ is a maximal order in $B$.
\end{lemma}

\begin{proof}
    An easy computation shows that $\mathrm{disc}(\mathcal{O})=\frac{1}{2}\left(-\left(\frac{m}{2}+2\right)r_m^3-6r_m^2+\left(\frac{m}{2}-2\right)r_m \right)$, which is a unit in $F_m$.
\end{proof}

\section{Lattices from quaternion algebras over totally real number fields}\label{1st}

    We will see how to define a lattice from an ideal  $\mathcal{I}$ of a quaternion algebra $B$ defined over a totally real number field.
    Let $F$ denote a totally real field, let $\mathbb{Z}_F$ denote the ring of integers of $F$ and let $\{\alpha_1,\ldots,\alpha_n\}$ be a $\mathbb{Z}$-basis of $\mathbb{Z}_F$.
    Let $\alpha$ denote a totally real element in $F$. We define the twisted embedding
   $$\begin{array}{rccl}
        \sigma_{B,\alpha}: & B & \rightarrow & \mathbb{R}^{4n} \\
        & x+yi+zj+tk & \mapsto & (\sqrt{2\sigma_1(\alpha)}\sigma_1(x),\ldots,\sqrt{2\sigma_n(\alpha)}\sigma_n(x), \\
        & & & \ldots,\sqrt{2\sigma_1(\alpha)}\sigma_1(t),\ldots,\sqrt{2\sigma_n(\alpha)}\sigma_n(t)).
   \end{array}$$
    Let $\mathcal{I}$ be an integral ideal of $B$ with $\mathbb{Z}_F$-basis
    $\left\{\beta_1,\beta_2,\beta_3,\beta_4\right\}$.
   
    \begin{definition}
        We define $\Lambda_{(\mathcal{I},\alpha)}:=\sigma_{B,\alpha}(\mathcal{I})$ to be the \textit{lattice given by $\mathcal{I}$ and twisted by $\alpha$}, or simply \textit{twisted ideal lattice}.
    \end{definition}

    In \cite{Costa2020} the authors show that $\Lambda_{(\mathcal{I},\alpha)}$ is a lattice with $\mathbb{Z}$-basis $\{\sigma_{B,\alpha}(\alpha_r\beta_s)\}$. 
    Let $\beta_s=\beta_{s1}+\beta_{s2}i+\beta_{s3}j+\beta_{s4}k$, for $1\leq s\leq 4$.
    The generator matrix of $\Lambda_{(\mathcal{I},\alpha)}$ is given by
    \begin{eqnarray}\label{gen_mat}
   M_\alpha=\left(
    \begin{smallmatrix}
        \sqrt{2\sigma_1(\alpha)}\sigma_1(\alpha_1\beta_{11}) & \ldots & \sqrt{2\sigma_n(\alpha)}\sigma_n(\alpha_1\beta_{11}) & \ldots &  \sqrt{2\sigma_1(\alpha)}\sigma_1(\alpha_1\beta_{14}) & \ldots & \sqrt{2\sigma_n(\alpha)}\sigma_n(\alpha_1\beta_{14}) \\
        \vdots & \ddots & \vdots & & \vdots & & \vdots \\
        \sqrt{2\sigma_1(\alpha)}\sigma_1(\alpha_n\beta_{11}) & \ldots & \sqrt{2\sigma_n(\alpha)}\sigma_n(\alpha_n\beta_{11}) & \ldots &  \sqrt{2\sigma_1(\alpha)}\sigma_1(\alpha_n\beta_{14}) & \ldots & \sqrt{2\sigma_n(\alpha)}\sigma_n(\alpha_n\beta_{14}) \\
        \vdots & & \vdots & \ddots & \vdots & & \vdots \\
        \sqrt{2\sigma_1(\alpha)}\sigma_1(\alpha_1\beta_{41}) & \ldots & \sqrt{2\sigma_n(\alpha)}\sigma_n(\alpha_1\beta_{41}) & \ldots &  \sqrt{2\sigma_1(\alpha)}\sigma_1(\alpha_1\beta_{44}) & \ldots & \sqrt{2\sigma_n(\alpha)}\sigma_n(\alpha_1\beta_{44}) \\
        \vdots & & \vdots & & \vdots & \ddots & \vdots \\
        \sqrt{2\sigma_1(\alpha)}\sigma_1(\alpha_n\beta_{41}) & \ldots & \sqrt{2\sigma_n(\alpha)}\sigma_n(\alpha_n\beta_{41}) & \ldots &  \sqrt{2\sigma_1(\alpha)}\sigma_1(\alpha_n\beta_{44}) & \ldots & \sqrt{2\sigma_n(\alpha)}\sigma_n(\alpha_n\beta_{44})
   \end{smallmatrix}\right)
   \end{eqnarray}
   
   In \cite{Costa2020} the authors prove the following formula for the volume of a twisted ideal lattice.

   \begin{prop}\label{volume}
        Let $F$ be a totally real number field with discriminant $d_F$.
        Let $B$ be a quaternion algebra over $F$ and $\mathcal{O}$ a maximal order in $B$. Let $\mathcal{I}$ be a right ideal of $\mathcal{O}$ and $\alpha$ a totally positive element in $F$. Then
        \begin{eqnarray}\label{det-formula}
              \det(\Lambda_{(\mathcal{I},\alpha)})=\mathrm{Nm}_{F/\mathbb{Q}}(\alpha)^2 d_F^2\mathrm{Nm}_{F/\mathbb{Q}}(\mathrm{Nm}_B(\mathcal{I}))^2\mathrm{Nm}_{F/\mathbb{Q}}(D_\mathcal{O}).
        \end{eqnarray}
    \end{prop}

   In \cite{Tu-Yang}, the authors give an equivalent construction to these ideal lattices \cite{Costa2020}*{Prop. 4.1}, using a certain quadratic form associated to $\mathcal{I}$ and $\alpha$. However, one cannot obtain a generator matrix for the lattices obtained from this construction. Instead, they show how to directly construct the Gram matrix of the lattice without using the embedding $\sigma_{B,\alpha}$.
    
    The inner product $\langle\cdot,\cdot\rangle$ on the lattice $\Lambda_{(\mathcal{I},\alpha)}$ is given by
    \begin{eqnarray}\label{inner}
        \langle x,y\rangle=\mathrm{Tr}_{F/\mathbb{Q}}(\alpha\,\mathrm{Tr}_{B}(x\overline{y})),\quad \mbox{for}\ x,y\in\Lambda_{(\mathcal{I},\alpha)}.
    \end{eqnarray}
   
    \begin{prop}
        Let $B$ be a quaternion algebra over a totally real field $F$ of degree $n$, $\mathcal{O}$ a maximal order in $B$ and $\mathcal{I}$ an integral right-ideal of $\mathcal{O}$. The lattice given by $\mathcal{I}$ is an integral even lattice of dimension $4n$.
   \end{prop}

    \begin{proof}
        According to (\ref{inner}) we have $||x||^2=\langle x,\overline{x}\rangle=\mathrm{Tr}_{F/\mathbb{Q}}(\alpha\mathrm{Tr}_B(\mathrm{Nm}_B(x)))=2\mathrm{Tr}_{F/\mathbb{Q}}(\alpha\mathrm{Nm}_B(x))$.
        Note that since $\alpha\in F$, $\mathrm{Tr}_{F/\mathbb{Q}}(\alpha\,\mathrm{Nm}_B(x))$ is not necessarily an integer.
        Clearly, if $\alpha=1$, then $\mathrm{Tr}_{F/\mathbb{Q}}(\mathrm{Nm}_B(x))\in\mathbb{Z}$. Hence, $||x||^2\in 2\mathbb{Z}$ for any $x\in\Lambda_\mathcal{I}$. Indeed, this property will hold for any $\alpha$ in the inverse of the different ideal (see the next section). 
    \end{proof}

In \cite{semisimple}, Bayer and Martinet gave a characterization of the lattices arising from quaternion algebras over totally real number fields.

\begin{theorem}\label{Bayer-Martinet}
Let $B$ be a quaternion algebra over a totally real number field $F$ with a maximal order $\mathcal{O}$ and $\mathcal{I}$ a left ideal in $\mathcal{O}$. The discriminants $\det(\Lambda_{(\mathcal{I},\alpha)})$ are of the form
\begin{eqnarray}\label{det-bayer-martinet}
\mathrm{Nm}_{F/\mathbb{Q}}(\mathfrak{p}_1\dots\mathfrak{p}_r)^2\mathrm{Nm}_{F/\mathbb{Q}}(\mathfrak{a})^4,
\end{eqnarray}
where $\mathfrak{p}_1\dots\mathfrak{p}_r$ are prime ideals in $F$ that are ramified in $F/B$ and $\mathfrak{a}$ is an ideal in $F$ equivalent to the different ideal in $F$ modulo squares in the narrow class group.
\end{theorem}
Theorem \ref{Bayer-Martinet} gives a way to construct unimodular lattices from quaternion algebras. In the next section we will further explore this result and present explicit constructions of some known lattices over infinitely many totally real number fields.

\section{Extremal even unimodular lattices from infinitely many quaternion algebras}\label{sec:extremal}


In this section  use the Landau symbols $O$ and $o$ with their usual meanings.
It is well known that even unimodular lattices (\emph{i.e.}, integral even lattices of determinant $1$) exist only in dimensions $8k$ where $k$ is a positive integer.
Furthermore, if $\Lambda$ is an $n$-dimensional even unimodular lattice then
\begin{equation}\label{extremal}
   \lambda_1(\Lambda)\leq 2+2 \left\lfloor\frac{n}{24}\right\rfloor. 
\end{equation}
Even unimodular lattices with the length of the shortest vector achieving the bound in (\ref{extremal}) are called \textit{extremal}. 
The following table summarises some known results about extremal even unimodular lattices.  
\begin{table}[ht]
    \centering
    \begin{tabular}{|c|c|c|c|c|c|c|c|c|}
    \hline
    $n$  & 8 & 16 & 24 & 32 & 40 & 48 & 72 & 80       \\ \hline
    $\lambda_1(\Lambda)$  & 2 & 2  & 4  & 4 & 4 & 6 & 8 & 8  \\ \hline
    \begin{tabular}[c]{@{}c@{}}Number \\ of extremal \\ lattices\end{tabular} & 1 & 2  & 1  & $\geq 10^{7}$ & $\geq 10^{51}$ & $\geq 5$ & $\geq 1$ & $\geq 4$ \\ \hline
    \end{tabular}
    \caption{Extremal even unimodular lattices}\label{ext}
\end{table}

    In this section we will use the lattice construction from Section \ref{1st}.
    Note that the lattices constructed in Section \ref{1st} are of dimension $4n$, thus, the degree $n$ should necessarily be even in order to obtain even unimodular lattices using this method.
    
    Before proceeding with our construction, we will need a few definitions from algebraic number theory.
    We consider $F$ to be a totally real number field of even degree $n$, discriminant $d_F$, $\alpha$ totally positive in $F$, and $\mathcal{O}$ a maximal order in $B=\left(\frac{-1,-1}{F}\right)$.
    We define the \textit{dual} of $\mathbb{Z}_F$ to be 
    $$\mathbb{Z}_F^{\vee}=\left\{x\in F~:~\mathrm{Tr}_{F/\mathbb{Q}}(x\mathbb{Z}_F)\subset \mathbb{Z}\right\}.$$
    By construction, if $\alpha\in \mathbb{Z}_F^{\vee}$, then the lattice $\Lambda_{(\mathcal{O},\alpha)}$ is integral and consequently even. If we additionally choose $\alpha$ such that $ \det(\Lambda_{(\mathcal{O},\alpha)})=1$, then the lattice $\Lambda_{(\mathcal{O},\alpha)}$ will be clearly even unimodular.
    Using Proposition \ref{volume}, we get
    $$\det(\Lambda_{(\mathcal{O},\alpha)})=\mathrm{Nm}_{F/\mathbb{Q}}(\alpha)^2 d_F^2\mathrm{Nm}_{F/\mathbb{Q}}(D_\mathcal{O}).$$
    Assume that $n>2$ is even and let $n_2=[F_\mathfrak{p}:\mathbb{Q}_2]$, where $\mathfrak{p}$ is a prime above $2$. Then, by Proposition \ref{discB}, we have $D_B=\mathbb{Z}_F$ if and only if $n$ is a power of two or $n_2$ is even. In this case, 
    $$\det(\Lambda_{(\mathcal{O},\alpha)})=\mathrm{Nm}_{F/\mathbb{Q}}(\alpha)^2 d_F^2.$$
    Hence, the lattice $\Lambda_{(\mathcal{O},\alpha)}$ is even unimodular if the norm of $\alpha$ is  $1/d_F$.
      
      Indeed, Theorem \ref{det-bayer-martinet} shows that it is necessary that $B$ is ramified at the finite places of $F$ so that the lattices obtained from a maximal order in  $B$ are unimodular. Note that this condition implies that $[F:\mathbb{Q}]$ is even. Moreover, we need a totally positive element $\alpha\in F$ such that $\alpha$ generates an ideal of the form $\mathrm{Nm}_{B}(\mathcal{I})(\mathbb{Z}_F^{\vee})^{-1}$, where $\mathcal{I}$ is a left ideal in $\mathcal{O}$. In the following, we will investigate these conditions in order to obtain explicit unimodular lattices.
    
    The dual $\mathbb{Z}_F^{\vee}$ is a fractional ideal in $F$ with $\mathbb{Z}_F\subset \mathbb{Z}_F^{\vee}$. Thus, its inverse is an integral ideal in $\mathbb{Z}_F$.
    The \textit{different ideal} of $F$ is defined as 
    the integral ideal $\mathfrak{D}_F:=(\mathbb{Z}_F^{\vee})^{-1}$.
    An important property of the different ideal is that 
    $\mathrm{Nm}_{F/\mathbb{Q}}(\mathfrak{D}_F)=d_F$.
    Now let $F=\Q(\beta)$ and $f$ be the minimal polynomial of $\beta$ in $\Q[X]$. It is well known that $(f'(\beta))$, the ideal generated by the derivative of $f$, is a subset of $\mathfrak{D}_F$. Furthermore, 
     \begin{equation}\label{fundeq}
         \mathrm{Nm}_{F/\mathbb{Q}}(f'(\beta))=(-1)^{\frac{n(n-1)}{2}}[\Z_F:\Z[\beta]]^2d_F.
     \end{equation}
     Let us first investigate the case where $[\Z_F:\Z[\beta]]=1$. In this case, the number field $F$ is said to be  $\textit{monogenic}$. In other words, the basis $\{1,\beta,\beta^2,\dots,\beta^{n-1}\}$ generates $\Z_F$. Such a basis is called a \emph{power basis}. Quadratic and cyclotomic fields are examples of such number fields.
    The literature on monogenic fields is extensive, and deciding which fields are monogenic is a long standing problem in algebraic number theory. For a detailed survey on the topic see \cite{gaal2002diophantine}.
     A classical argument shows that,
    if $F$ is monogenic with generator $\beta$ and minimal polynomial $f$, then
    $\mathfrak{D}_F=(f'(\beta))$.
    Consequently,
    $\mathrm{Nm}_{F/\mathbb{Q}}(f'(\beta))=d_F$.
    We will use this fact to describe a strategy to construct even unimodular lattices from quaternion algebras over totally real number fields of even degree such that  $\mathbb{Z}_F=\mathbb{Z}[\beta]$ for some $\beta\in\mathbb{Z}_F$ with minimal polynomial $f$.

    \begin{theorem}\label{power2}
        Let $F$ be a totally real monogenic number field of even degree $n=2^t$, $t>0$, or $n_2=[F_\mathfrak{p}:\mathbb{Q}_2]$ even, where $\mathfrak{p}$ is a prime above $2$.
        Let $\varepsilon\in\mathbb{Z}_F$ be a unit such that
        $\varepsilon\cdot f'(\beta)$ is totally positive.
        Then the lattice $\Lambda_{(\mathcal{O},\alpha)}$ is an even unimodular lattice for $\alpha=\frac{1}{\varepsilon\cdot f'(\beta)}$,
        where $\mathcal{O}$ is any maximal order in $B=\left(\frac{-1,-1}{F}\right)$.
    \end{theorem}
    \begin{proof}
    Since
    $\mathbb{Z}_F$ is monogenic, we have that $\mathfrak{D}_F=(f'(\beta))$, where $f'$ is the derivative of $f$.
    If $[F:\mathbb{Q}]=2^t$, $t>0$, or $n_2=[F_\mathfrak{p}:\mathbb{Q}_2]$ even, the discriminant of $B$ is $\mathbb{Z}_F$ according to Proposition \ref{discB}.
    
    Using Proposition \ref{volume} we obtain
    $\det(\Lambda_{(\mathcal{O},\alpha)})=d_F^{2}\mathrm{Nm}_{F/\mathbb{Q}}(\alpha)^2$.
    In order to construct an unimodular lattice we need to find a totally positive element $\alpha\in \mathbb{Z}_F^{\vee}$ such that $\mathrm{Nm}_{F/\mathbb{Q}}(\alpha)=\frac{1}{d_F}$. By assumption, the element $\frac{1}{\varepsilon\cdot f'(\beta)}$ has these properties.
    \end{proof}
    \begin{remark}
    If $n_2$ is odd, then $D_B=2\mathbb{Z}_F$. In this case,
    $\det(\Lambda_{(\mathcal{O},\alpha)})=2^n d_F^2\mathrm{Nm}_{F/\mathbb{Q}}(\alpha)^2$.
    Consequently, to obtain an even integral lattice $\Lambda_{(\mathcal{O},\alpha)}$ it is  enough to find $\alpha\in\Z_F$ such that $\mathrm{Nm}_{F/\Q}(\alpha)=2^{\frac{n}{2}}$ and $\alpha\varepsilon f(\beta)'$ is totally positive, for some unit $\varepsilon\in\Z_F$.
    \end{remark}
    
In the following, we will explicitly construct infinite families of even unimodular lattices using the above construction. It is proved in \cite{semisimple} that it is impossible to construct the Leech using the same construction. In fact, the authors proved the following theorem.
\begin{theorem}
    Even unimodular lattices of dimension $n<48$ obtained from quaternion algebras are of minimal norm $2$. In particular, it is impossible to construct the Leech lattice from a quaternion algebra over a totally real number field of degree $6$.
\end{theorem}

\subsection{Dimension 8}\cami{This paragraph again makes it sound like nothing new happens in this section. Can you emphasize the contribution a bit here?}
 Lattices obtained from quaternion algebras that are equivalent to $E_8$ are characterised in \cite{semisimple}. In this section, we will present an infinite family of lattices similar to $E_8$ from infinitely many quaternion algebras with explicit bases for the underlying maximal orders.

    Theorem \ref{power2} covers number fields of degree $n>2$. We will investigate quadratic number fields and use a similar strategy  to prove a density result for the set of number fields for which $\Lambda_{(\mathcal{O},\alpha)}$ is the $E_8$ lattice.
 It is well known \cite{Conway-Sloane}*{Ch. 4} that the lattice $E_8$ is the only even unimodular lattice of dimension $8$ up to isometry. Hence, to construct $E_8$ it is enough to find a quadratic field $F$ with $D_B=\mathbb{Z}_F$ and $F$ contains a totally positive element $\alpha$ such that $\mathrm{Nm}_{F/\mathbb{Q}}(\alpha)=\frac{1}{d_F}$.

    \begin{theorem}\label{thm-E8}
        Let $D\equiv 1 \pmod4$ be a square-free positive integer of the form $D=s^2+4$ and let $F=\mathbb{Q}(\sqrt{D})$ denote a real quadratic field.
        Let $B=\big(\frac{-1,-1}{F}\big)$ denote the quaternion algebra over $F$ and let $\mathcal{O}$ denote the maximal order in $B$ given by Lemma \ref{d1}.
        Let $\alpha=\frac{\varepsilon}{\sqrt{D}}$, with $\varepsilon = \frac{s+\sqrt{D}}{2}$.
        Then the lattice $\Lambda_{(\mathcal{O},\alpha)}$ is similar to $E_8$ and its generator matrix is given by $(\ref{gen_mat})$.
    \end{theorem}

    \begin{proof}
        Since $D\not\equiv 1\pmod8$, according to Prop. \ref{discB} we have $D_B=\mathbb{Z}_F$, and
        according to Lemma \ref{d1}, we have that
        $$\mathcal{O}=\mathbb{Z}_F\left[1,i,\frac{\sqrt{D}+1}{4}+\frac{\sqrt{D}+3}{4}i+\frac{j}{2}, \frac{\sqrt{D}+3}{4}+\frac{\sqrt{D}+1}{4}i + \frac{k}{2}\right]$$
        is a maximal order in B.
        Using (\ref{det-formula}), the lattice $\Lambda_{(\mathcal{O},\alpha)}$ twisted by some totally real $\alpha\in F$ has
        $\det(\Lambda_{(\mathcal{O},\alpha)}) = d_F^4 Nm_{F/\mathbb{Q}}(\alpha)^4$,
        where $d_F$ is the discriminant of $F$.
        After the previous remark, for the lattice $\Lambda_{(\mathcal{O},\alpha)}$ to be similar to $E_8$, we need to find a totally positive element $\alpha\in F$ of norm $\mathrm{Nm}_{F/\mathbb{Q}}(\alpha)=\frac{1}{d_F}=\frac{1}{d}$ and with $\mathrm{Tr}_{F/\mathbb{Q}}(\alpha \mathrm{Nm}_B(x))\in\mathbb{Z}$ for all $x\in\mathcal{O}$.
        
        Consider $\varepsilon = \frac{s+\sqrt{D}}{2}$. It is easy to check that $\varepsilon$ is a unit in $\mathbb{Z}_F$ of norm $-1$. We can assume that $\varepsilon>0$ without loss of generality, as we can replace $\varepsilon$ by $-\varepsilon$ if necessary.
        Take $\alpha=\frac{\varepsilon}{\sqrt{D}}$. 
        By assumption we have that
        $\mathrm{Nm}_F(\varepsilon) = \varepsilon\overline{\varepsilon} = -1$, 
        thus $\overline{\varepsilon}<0$ and $\overline{\alpha}=\frac{\overline{\varepsilon}}{-\sqrt{D}}>0$. 
        Hence, $\alpha$ is totally positive and $\mathrm{Nm}_F(\alpha) = \frac{1}{D}$.
        Furthermore, $\mathrm{Tr}_{F/\mathbb{Q}}(\alpha)=1$, so $\mathrm{Tr}_{F/\mathbb{Q}}(\alpha \mathrm{Nm}_B(x))\in\mathbb{Z}$.
        
        For the case $D|\not\equiv 1\pmod 4$, we consider $D=s^2+1$, $\varepsilon=s+\sqrt{D}$ and proceed in a similar way.
    \end{proof}
    
It is shown in \cite{ester} that the set of integers $s$ such that $s^2+4$ is square-free has a positive density. Hence the following corollary.

\begin{corollary}\label{E88}
\cami{Should we be self-contained and add (here and in other similar thms): Let $\alpha$...} 
The set
    $$\mathcal{T}_{E_8}=\left\{ F \textrm{ quadratic totally real}\ \Big| \ \Lambda_{(\mathcal{O},\alpha)}\sim E_8\textrm{ where } \mathcal{O} \subseteq  B=\left(\frac{-1,-1}{F}\right) \textrm{ is a maximal order}\right\}$$
    is infinite.

\end{corollary}
\begin{remark}
   The infinite set in Corollary \ref{E88} is explicitly given by $\Q(\sqrt{D})$ and $\alpha$ as in the proof of Theorem \ref{thm-E8}.
\end{remark}

We will further obtain a lower bound on the number of number fields $F$ from which the $E_8$ lattice can be realised from a maximal order in $B=\big(\frac{-1,-1}{F}\big)$.
Let $F=\mathbb{Q}(\sqrt{D})$ be a real quadratic field, with $D>1$ a positive square-free integer. 
     We have $\mathfrak{D}_F=(t)$, with $t=2\sqrt{D}$ if $D\not\equiv 1\pmod4 $, or $t=\sqrt{D}$ if $D\equiv 1\pmod4$. Clearly, $t$ is not totally positive in both cases, thus, to apply a similar argument to the one in Theorem \ref{power2}, we need $F$ to contain a unit $\varepsilon$ of norm $-1$. This is equivalent to the negative Pell equation 
     \begin{equation}\label{pell}
         x^2-d_F y^2=-1
     \end{equation}
     being solvable in $\Z^2$. Note that if $p\equiv 3\pmod4$ for some $p\mid d_F$, then Equation (\ref{pell}) is not solvable, otherwise, $-1$ will be a square modulo $p$.
     In the following, we will reduce the problem of counting number fields $F$ such that $\Lambda_{(\mathcal{O},\alpha)}\sim E_8$ to counting the number of discriminants $d_F$ in some given set.
     A folklore theorem due to Hermite--Minkowski states that for a fixed degree $n$, there are only finitely many number fields of discriminant $d_F\leq X$ for any $X>1$. Thus, it is canonical to order number fields with respect to their discriminant.
     
     \tao{Define D2, check and fix things here.}
     We denote by $\mathcal{N}^{-}$ the set of discriminants $d_F$ such that $F$ contains a fundamental unit of norm $-1$. \cami{Define fundamental discriminant?}
     Let $$\mathcal{D}_2=\{d_F ~|~F \textrm{ quadratic field, }p\not\equiv3\pmod 4 \textrm{ for all } p|d_F\}.$$ 
     Note that $\mathcal{D}_2$ is a disjoint union of the subsets 
     $$\mathcal{D}_{\textrm{even}}=\{d_F\in \mathcal{D}_2~|~ d_F\equiv 0\pmod8 \}$$
     and 
     $$\mathcal{D}_{\textrm{odd}}=\{d_F\in \mathcal{D}_2~|~ d_F\equiv 1\pmod4 \}.$$
     Using the above notation, we have the following lemma.
     
  \begin{lemma}\label{fundalemma8}
      Let $F=\mathbb{Q}(\sqrt{D})$ be a real quadratic field such that $d_F\in \mathcal{D}_{\textrm{even}}\cap \mathcal{N}^{-} $ or $d_F\in \mathcal{D}_{\textrm{odd}}\cap \mathcal{N}^{-} $ with $d_F\not\equiv 1\pmod8$. Then
      there exists $\alpha\in F$ totally positive, such that $\Lambda_{(\mathcal{O},\alpha)}\sim E_8$ for  any maximal order $\mathcal{O}$ in $B=\left(\frac{-1,-1}{F}\right)$.
  \end{lemma}
  
  \begin{proof}
      Using Proposition \ref{discB} with $n=2$, we get that 
			$$D_B=\left\{\begin{array}{lll}
			   2\mathbb{Z}_F & \mathrm{if} & D\equiv 1\quad (\mathrm{mod}\ 8), \\
			    \mathbb{Z}_F & \mathrm{if} & D\not\equiv 1\quad (\mathrm{mod}\ 8).
			   \end{array}\right.$$
   Clearly, if $d_F\in \mathcal{D}_{\textrm{even}}\cap \mathcal{N}^{-} $ or $d_F\in \mathcal{D}_{\textrm{odd}}\cap \mathcal{N}^{-} $ with $d_F\not\equiv 1\pmod8$, then 
   $$\mathrm{Nm}_{F/\mathbb{Q}}(D_\mathcal{O})=1.$$ 
  Consequently, 
    $$  \det(\Lambda_{(\mathcal{O},\alpha)})=\mathrm{Nm}_{F/\mathbb{Q}}(\alpha)^2 d_F^2$$
    for any maximal order $\mathcal{O}\in B$.
    
     On the other hand, we have $\mathfrak{D}_F=(t)$, with $t=2\sqrt{D}$ if $D\not\equiv 1\pmod4 $, or $t=\sqrt{D}$ if $D\equiv 1\pmod4 $. 
       By assumption, $F$ contains a unit $\varepsilon$ of norm $-1$.
    If $\varepsilon<0$, then its conjugate is positive, replacing $\varepsilon$ by $-\varepsilon$ if necessary, we get $\varepsilon t$ is totally positive.
    Thus, for $\alpha=\frac{1}{\varepsilon t}$, we have $\alpha\in\Z_F^{\vee}$ totally positive of norm $\frac{1}{d_F}$.
    Finally, for $\alpha=\frac{1}{\varepsilon t}$, we have that $\Lambda_{(\mathcal{O},\alpha)}$ is an even unimodular lattice of dimension $8$.
  \end{proof}
  
     Denoting by
     $$\mathcal{T}_{E_8}(X)=\#\left\{F \mathrm{\ quadratic \ totally\ real }\ \Big|\ d_F\leq X\,,   \Lambda_{(\mathcal{O},\alpha)}\sim E_8, \textrm{ where } \mathcal{O} \subseteq B=\left(\frac{-1,-1}{F}\right) \textrm{ is a maximal order} \right\} $$
  we obtain the following corollary. 
  
    \begin{corollary}\label{E_8}
    For $X\rightarrow\infty$, we have the inequality
    $$(\tau -o(1))c_1\frac{X}{\sqrt{\log(X)}}\leq \mathcal{T}_{E_8}(X),$$
    where  $c_1=\frac{9}{8\pi}\prod_{p\equiv 1\pmod 4}(1-p^{-2})^{1/2}$ and  $\tau=\frac{1}{9}\prod_{j\textrm{ odd}}(1-2^{-j})=0,046602493\dots$
    \end{corollary}
    \begin{proof}
     Let $X>1$. We denote by $\mathcal{D}(X)$ the cardinality of $\mathcal{D}\cap [1,X]$.
     It is known  \cite{weil} that 
     $$\mathcal{D}_2(X)\sim c_1\frac{X}{\sqrt{\log(X)}},$$
     $$\mathcal{D}_{\textrm{odd}}(X)\sim \frac{8}{9}c_1\frac{X}{\sqrt{\log(X)}},$$
     and 
     \begin{equation}\label{eqeven}
         \mathcal{D}_{\textrm{even}}(X)\sim\frac{1}{9} c_1\frac{X}{\sqrt{\log(X)}}.
     \end{equation}
    By a straightforward application of Lemma \ref{fundalemma8}, we obtain
    $$(\mathcal{D}_{\textrm{even}}\cap\mathcal{N}^{-1})(X)\leq \mathcal{T}_{E_8}(X).$$
    It is proved in \cite{fouvry2010negative} that for $X\rightarrow\infty$, we have
    $$(\beta -o(1))(\mathcal{D}_{\textrm{even}}(X)\leq (\mathcal{D}_{\textrm{even}}\cap\mathcal{N}^{-1})(X)\leq \left(\frac{3}{2} +o(1)\right)\mathcal{D}_{\textrm{even}}(X)$$  
    where 
    $$\beta=\prod_{j\textrm{ odd}}(1-2^{-j}).$$
    The result now follows from \eqref{eqeven}.
    \end{proof}    

    As a consequence of the previous results, we obtain an infinite family of quaternion algebras from which one can construct the $E_8$ lattice.
    \begin{remark}
       Corollary \ref{E_8} gives a lower bound on the density of the set of quaternion algebras $B$ where the $E_8$ lattice is realisable as an embedding of a maximal order in $B$. Here we ordered the algebras with respect to their underlying number fields.
       As mentioned previously, we order number fields with respect to the discriminant.  If we restrict our (discriminants) counting function to $\mathcal{D}_2$ (instead of all integers), then the result means that at least $49\%$ of the discriminants in $\mathcal{D}_2$ are discriminants of number fields $F$ with $B=\left(\frac{-1,-1}{F}\right)$ in $\mathcal{T}_{E_8}$. It is conjectured that this percentage is precisely $58\%$ (see for instance \cite{fouvry2010negative}).
    \end{remark}
    
	\begin{example}
	Let $F=\mathbb{Q}(\sqrt{5})$. We have $d_F=5$, clearly $d_F\in \mathcal{D}_{\textrm{odd}}\cap \mathcal{N}^{-} $ with $d_F\not\equiv 1\pmod8$.
	In fact,  $\varepsilon=2+\sqrt{5}$ has norm $-1$.
	Moreover, $\Z_F$ is generated by $\frac{1+\sqrt{5}}{2}$, and $\mathfrak{D}_F=(\sqrt{5})$. Furthermore, $\alpha=\frac{1}{\varepsilon\sqrt{5}}$ is totally positive of norm $\frac{1}{d_F}$.
	By Proposition \ref{d1},
	$$\mathcal{O}=\mathbb{Z}_F\left[1,i,\frac{\sqrt{5}+1}{4}+\frac{\sqrt{5}+3}{4}i+\frac{j}{2},\frac{\sqrt{5}+3}{4}+\frac{\sqrt{5}+1}{4}i+\frac{k}{2}\right]$$
	is a maximal order in $B=\left(\frac{-1,-1}{F}\right)$.
    Consequently, the lattice $\Lambda_{(\mathcal{O},\alpha)}$ is even unimodular. Thus, it is similar to $E_8$. Its associated Gram matrix is
    $$\left(\begin{smallmatrix}
    4& -2 & 0 & 0& -1 & 1&  1&  0\\
    -2&  2&  0&  0&  1&  0&  0&  1\\
     0&  0&  4& -2&  1&  0& -1&  1\\
     0&  0& -2&  2&  0&  1 & 1&  0\\
    -1&  1&  1&  0&  2&  0  &0&  1\\
     1&  0&  0&  1&  0&  2  &1&  1\\
     1&  0& -1&  1&  0&  1  &2&  0\\
     0&  1&  1&  0&  1&  1  &0&  2\\
    \end{smallmatrix}\right).$$
	\end{example}

\subsection{Dimension 16}\label{subsec:dim16}
    
    The argument that we used to prove Theorem \ref{E_8} is based on the monogenity of $F$. Unfortunately, not all number fields have this property.
    Indeed, it was proved in \cite{gras1985condition} that with the exception of the maximal real subfields of cyclotomic fields, abelian extensions of $\mathbb{Q}$ of prime degree greater than or equal to $5$ are not monogenic. Moreover, Gras \cite{gras1986non} showed that almost all abelian extensions of $\mathbb{Q}$ with degree coprime to $6$ are not monogenic.
    In this section we will show a result analogous to Theorem \ref{E_8} in dimension $16$ by investigating an infinite family of non-monogenic real quartic number fields. 
    
Note that in the previous section we used the fact that 
$$\mathfrak{D}_F=(f'(\beta)).$$
 In general, we have that $$(f'(\beta))\subset \mathfrak{D}_F, $$
where $F=\Q(\beta)$ and $\beta\in\Z_F$.
Hence,
$$\Z_F\subset\Z_F^{\vee}\subset \frac{1}{f'(\beta)}\Z_F.$$
Thus, assuming that $[\Z_F:\Z[\beta]]$ is a square, by Proposition \ref{fundeq}, the norm of
$$\alpha=\frac{[\Z_F:\Z[\beta]]^{1/2}}{\varepsilon f'(\beta)}$$
is
$$|\mathrm{Nm}_{F/\mathbb{Q}}(\alpha)|=\frac{[\Z_F:\Z[\beta]]^{n/2-2}}{d_F}$$
for any unit $\varepsilon\in\Z_F$.
Consequently, if $n=4$, then $\alpha\in \Z_F^{\vee}$.
Hence, if $\varepsilon$ is chosen such that $\alpha$ is totally positive, then the lattice  $\Lambda_{(\mathcal{O},\alpha)}$ is an even lattice of volume 
 $$  \det(\Lambda_{(\mathcal{O},\alpha)})=\mathrm{Nm}_{F/\mathbb{Q}}(D_\mathcal{O}).$$

    Table \ref{ext} shows that there are exactly two extremal even unimodular lattices (up to isometry) in dimension $16$. In fact, it is well known that these lattices are $E_8^2$ and $D_{16}$. We will use the above strategy to construct the extremal lattice $E_8^2$ and the densest known lattice packing $\Lambda_{16}$ from infinitely many quaternion algebras. For this purpose we will use the so called \textit{simplest quartic fields}.
    For any integer $m\in\mathbb{N}\setminus\{3\}$ such that the odd part of $m^2 +16$ is square-free, we consider the polynomial
    $$f_m(x)=x^4-mx^3-6x^2+mx+1$$
    of discriminant $d_{f_m}=4\Delta_m^3$, where $\Delta_m=m^2+16$. The polynomial $f_m$ defines different real cyclic quartic number
    fields $F_m$ for different values of $m$.
    The fields $F_m$ are called the simplest quartic fields \cite{stephane}. Let $r_m$ be the largest root of $f_m$. Then we can write $F_m=\mathbb{Q}(r_m)$. Let $\sigma$ be a generator of $\Gal(F_m/\mathbb{Q})$. One can show that \cite{stephane} $$\sigma(r_m)=\frac{r_m-1}{r_m+1}\,,\ \sigma(r_m)^2=-\frac{1}{r_m}\,,\  \textrm{ and }\ 
\sigma(r_m)^3=-\frac{r_m+1}{r_m-1}.$$
   The monogenity problem for the fields $F_m$ was completely solved in \cite{gaal2014calculating}. In fact, the authors showed that, with the exceptions of $m=2, 4$, the fields $F_m$ are not monogenic. Moreover, the authors showed that $[\Z_{F_m}:\Z[r_m]]=4, 8$ for $v_2(m)=1,2$ respectively, where $v_2(m)$ is the $2$-adic valuation of $m$. Furthermore, they present an integral basis for $F_m$ for all $m$. More precisely, assuming that $m^2+16$ is not divisible by an odd square, an integral basis of $F_m$ is given by
    \begin{align*}
        \left(1,r_m,r_m^2,\frac{1+r_m^3}{2}\right) ~~&\mathrm{ if } ~v_2(m)=0,\\
        \left(1,r_m, \frac{1+r_m^2}{2},\frac{r_m+r_m^3}{2}\right) ~~&\mathrm{ if }~ v_2(m)=1,\\
        \left(1,r_m, \frac{1+r_m^2}{2},\frac{1+r_m+r_m^2+ r_m^3}{4}\right) ~~&\mathrm{ if }~ v_2(m)=2,\\
        \left(1,r_m,\frac{1+2r_m-r_m^2}{4},\frac{1+r_m+r_m^2+ r_m^3}{4}\right) ~~&\mathrm{ if }~ v_2(m)\geq 3.\\
    \end{align*}
It is proved in \cite{gaal2014calculating} that if $v_2(m)=1$, then $[\Z_{F_m}:\Z[r_m]]=4$.
We consider $s=\frac{r_m^3 - mr_m^2 - 5r_m + m}{2}$. Using the above we can see that the element $\alpha = \frac{2}{sf_m(r_m)'}$ is totally positive of norm $\frac{1}{d_{F_m}}$. 
Consequently, the lattice $\Lambda_{(\mathcal{O},\alpha)}$ is even unimodular of dimension $16$ for any maximal order in a quaternion algebra over $F_m$. 
We have the following  result. 


\begin{theorem}\label{E16} 
\cami{Let $\alpha$....} 
$$\mathcal{T}_{E_8^2}=\left\{F \mathrm{\ quartic\ totally\ real  }\ \Big| \    \Lambda_{(\mathcal{O},\alpha)}\sim E_8^2, \textrm{ where } \mathcal{O} \subseteq B=\left(\frac{-1,-1}{F}\right) \textrm{ is a maximal order} \right\} $$
   is infinite.
\end{theorem}

\begin{proof}
    Let $F_m$ denote a simplest quartic field as described at the beginning of this section.
    We will show that, if $m=2k$ for some odd $k$, \emph{i.e.}, $v_2(m)=1$, then $\Lambda_{(\mathcal{O},\alpha)}$ is even unimodular for $\alpha = \frac{2}{sf_m(r_m)'}$.
    Moreover, we have $\Delta_m=4(k^2+4)$. It is proved in \cite{ester} that the set of square-free integers of the form $k^2+4$, with $k$ odd, has a positive density.
    We consider the quaternion algebra $B=\left(\frac{-1,-1}{F_m}\right)$.
    The basis
    $$\mathcal{B}=\left\{1, \frac{(1+r_m)(1+i)}{2}, \frac{(1+r_m)(1+j)}{2}, \frac{1+i+j+k}{2}\right\}$$ 
    is a basis of a maximal order in $B$ (indeed, the discriminant of the basis $\mathcal{B}$ coincides with $D_B$).
    One can construct the lattice associated to this order explicitly, using the integral basis for $F_m$ in the case $v_2(m)=1$, and compare the densities of $E_8^2$ and $\Lambda_{(\mathcal{O},\alpha)}$.
  
\end{proof}

If $v_2(m)=2$, then by equation \ref{fundeq}, we have

\begin{equation}
    \mathrm{Nm}_{F_m/\mathbb{Q}}(f_m(r_m)')=2^6 d_{F_m}.
\end{equation}
Hence, 
$$  \mathrm{Nm}_{F_m/\mathbb{Q}}\left(\frac{f_m(r_m)'}{4}\right)= \frac{d_{F_m}}{4}.$$
We consider $s=\frac{r_m^3 - mr_m^2 - 5r_m + m}{4}$. Using the above we can see that $\mathrm{Nm}_{F_m/\mathbb{Q}}(s)=1$.
Moreover, the element $\alpha = \frac{4}{sf_m(r_m)'}$ is totally positive of norm $\frac{2^4}{ d_{F_m}}$. 
Consequently, taking  $\alpha=\frac{2}{s f(r_m)'}$, the lattice $\Lambda_{(\mathcal{O},\alpha)}$ is an even lattice of volume $2^8$ with minimal norm $4$. A lattice with such properties is necessarily isometric to the Barnes–Wall lattice $\Lambda_{16}$, which is the best known lattice sphere packing in its respective dimension.

We can check that the basis $\mathcal{B}$ in the proof of Theorem \ref{E16} generates a maximal order in $\left(\frac{-1,-1}{F_m}\right)$. We arrive at the following theorem, which provides an analogous result for the best known sphere packing in dimension $16$.

\begin{theorem} The set 
$$\mathcal{T}_{\Lambda_{16}}=\left\{ F \mathrm{\ quartic~totally\ real } \ \Big| \   \Lambda_{(\mathcal{O},\alpha)}\sim \Lambda_{16}, \textrm{ where } \mathcal{O} \subseteq B=\left(\frac{-1,-1}{F}\right) \textrm{ is a maximal order} \right\} $$
   is infinite.
\end{theorem}

\begin{proof}
    The proof is analogous to the proof of Theorem \ref{E16}. 
    Note that if $m=4k$ with $k$ odd, then we need the odd part in $\Delta_m=16(k^2+1)$ to be square-free. Hence, the conclusion is implied by the fact that the set of integers $k$ such that $k^2+1$ is square-free is of a positive density \cite{ester}. 
\end{proof}

\begin{example}
    We consider $m=20$. The quartic number field $F_{20}$ is given by the polynomial $$f_{20}(x)=x^4-10x^3-6x^2+20x+1.$$
    We have $F_{20}=\mathbb{Q}(r_m)$ is of discriminant $d_{F_{20}}=4(m^2+16)^3$, where
    $$r_m=\frac{1}{2}\left( \frac{m+\sqrt{\Delta_m}}{2}+\frac{\sqrt{\Delta_m+m\sqrt{\Delta_m}}}{2}\right).$$
    Taking $s=\frac{r_m^3 - mr_m^2 - 5r_m + m}{2}$, the element $\alpha = \frac{1}{sf_m(r_m)'}$ yields an even lattice $\Lambda_{(\mathcal{O},\alpha)}$ similar to the lattice $\Lambda_{16}$, with Gram matrix
    
    $$\begin{pmatrix}
    4& 0& 0& 2& 2& 2& 0& 4& 2& 2& 0& 4& 2& 0& 0 & 1 \\
    0& 4& 0& 6& 2& 2& 4& 46& 2& 2& 4& 46& 0& 2& 0& 3\\
    0& 0& 4& 42 & 0& 4& 42& 448& 0& 4& 42& 448& 0& 0& 2& 21\\
    2& 6& 42& 452& 4& 46& 448& 4786& 4& 46& 448& 4786& 1& 3& 21& 226\\
    2& 2& 0& 4& 4& 4& 4& 50& 2& 2& 2& 25& 2& 2& 0& 4\\
    2& 2& 4& 46& 4& 12& 88& 946& 2& 6& 44& 473& 2& 2& 4& 46\\
    0& 4& 42& 448& 4& 88& 892& 9522& 2& 44& 446& 4761& 0& 4& 42& 448\\
    4& 46& 448& 4786& 50& 946& 9522& 101660& 25& 473& 4761& 50830& 4& 46& 448& 4786\\
    2& 2& 0& 4& 2& 2& 2& 25 &4& 4& 4& 50& 2& 2& 0& 4\\
    2& 2& 4& 46& 2& 6& 44& 473& 4& 12& 88& 946& 2& 2& 4& 46\\
    0& 4& 42& 448& 2& 44& 446& 4761& 4 &88 &892& 9522& 0& 4& 42& 448\\
    4& 46& 448& 4786& 25& 473& 4761& 50830& 50& 946& 9522& 101660 & 4 &46& 448& 4786\\
    2& 0& 0& 1& 2& 2& 0& 4& 2& 2& 0& 4& 4& 0& 0& 2\\
    0& 2& 0& 3& 2& 2& 4& 46& 2& 2& 4& 46& 0& 4& 0& 6\\
    0& 0& 2& 21& 0 &4& 42& 448& 0& 4& 42& 448& 0& 0& 4& 42\\
    1& 3& 21& 226& 4& 46& 448& 4786& 4& 46& 448& 4786& 2 &6& 42& 452
    \end{pmatrix}\,.$$
\end{example}


\section{Potential applications and future work}\label{sec:apps}

While even unimodular lattices and dense lattice packings are interesting in their own right,  they can also be utilised in, \emph{e.g.}, code design for wireless communications and physical layer security  \cite{OgVi,OgViBe,Veh,Sethu,BeOg2011,Belfiore-MIMO,wiretapdamiralex}. A detailed introduction to the algebraic tools used to design codes from number fields and cyclic division algebras and the associated lattice code constructions is given in \cite{OgVi} for single-input single-output communications (SISO) and in \cite{OgViBe} for multiple-input multiple-output (MIMO) communications.

Typically, for SISO, vector lattices are utilised. Totally real number fields and their full-diversity ideal lattices arising via the canonical embedding have turned out to be beneficial in terms of reliable communications. In terms of security, well-rounded lattices with good packing density have been proposed for wiretap channels in \cite{wiretapdamiralex} to enhance security on the physical layer. 

Motivated by the use of ideal lattices in the design of lattice codes for reliability and security in fading channels, we are interested in lattices arising from a similar structure, but which differ from ideals in number fields in that they are non-commutative rings. Namely, we have resorted to ideals in quaternion algebras, hence in a way using ``MIMO machinery'' \cite{OgViBe,Veh,Sethu} to construct lattices  suitable for SISO communications. 

Depending on the channel quality, different metrics become relevant as design criteria. The lattices constructed in this paper are not of full diversity, which means they are mainly suitable for low--to--mid-range channel quality, where the density of the lattice is crucial. Such channel qualities are present in, \emph{e.g.}, long-distance or low-power communications, including within the Internet of Things (IoT), 5G networks and beyond.

As future work, we are interested in exploring more dimensions where we can potentially construct dense and interesting lattices from applications' point of view, to be utilised in, \emph{e.g.}, coding theory, cryptography, or machine learning. In order to obtain higher dimensional dense lattices with our method, we aim to introduce some machine learning techniques to help us find the elements that we need (\emph{e.g.}, bases of maximal orders in quaternion algebras or the twisting element $\alpha$). Similar approaches have been taken recently to explore the use of machine learning algorithms to learn properties about algebraic and number theoretic objects \cite{ML_NT}, \cite{ML_NF}, \cite{ML_math}.

\bibliography{References}
\bibliographystyle{alpha}

\newpage

\section*{Appendix}

In this section we provide the MAGMA source code for the construction used in this paper.

\begin{verbatim}
intrinsic basis_order(B :: AlgQuat) -> SeqEnum
{Input: Quaternion algebra over real quadratic field and a=b=-1. 
Output: Basis of maximal order in B.}

    ZZ := Integers();
    F := BaseField(B); d := ZZ!(F.1^2);
    I := B.1; J := B.2; K := B.3;
    a := ZZ!(B.1^2); b := ZZ!(B.2^2);
    if d eq 2 then
        basis := [1, (1+I)*F.1/2, (1+J)*F.1/2, (1+I+J+K)/2];
    else if d mod 4 eq 3 then
        basis := [1, I, (F.1*I+J)/2, (F.1+K)/2];
    else
        if d mod 8 eq 1 then
            basis := [1,I,J,(1+I+J+K)/2];
        else
            basis := [1, I, (F.1+1)/4+(F.1+3)/4*I+J/2, (F.1+3)/4+(F.1+1)/4*I+K/2];
        end if;
    end if; end if;
    return basis;
end intrinsic;


intrinsic Z_basis_order(basis :: SeqEnum) -> SeqEnum
{Input: basis of an integral ideal I in a maximal order in a quaternion algebra B 
over totally real field of degree d.
Output: basis of I as a Z-basis of length 4d.}

    F := BaseRing(basis[1]);
    ZFbasis := IntegralBasis(F);
    Zbasis := [alpha*beta : alpha in ZFbasis, beta in basis];
    return Zbasis;
end intrinsic;


intrinsic Gram_matrix(basis:: SeqEnum: alpha:=1) -> AlgMatElt
{Input: basis of an ideal I in a maximal order in a quaternion algebra B 
over totally real field F of degree d, and a totally real element alpha 
in F (alpha = 1 by default).
Output: Gram matrix of the Z-lattice associated to I, alpha.}
    
    Zbasis := Z_basis_order(basis);
    QQ := Rationals();
    F := BaseRing(Zbasis[1]);
    rows := [QQ!(Trace(alpha*Trace(x*Conjugate(y)))) : x,y in Zbasis];
    G := Matrix(QQ, #Zbasis, rows);
    return G;
end intrinsic;


intrinsic Z_lattice_matrix(basis:: SeqEnum: alpha:=1) -> AlgMatElt
{Input: basis of an ideal I in a maximal order in a quaternion algebra B 
over totally real field F of degree d, and a totally real element alpha in F 
(which is 1 by default).
Output: generator matrix of the Z-lattice associated to I, alpha.}

    Zbasis := Z_basis_order(basis);
    RR := RealField(50);
    F := BaseRing(Zbasis[1]); d := Degree(F);
    n := #Zbasis;
    rows_short := [Eltseq(x) : x in Zbasis];
        
    if d eq 1 then
        rows :=  rows_short[1] cat rows_short[2] cat rows_short[3] cat rows_short[4];
        matrix := Matrix(Rationals(), 4, rows);
    else
        rows := [];
        if alpha in Rationals() then
            rows := [Sqrt(2*alpha)*RealEmbeddings(row[i])[j] :
            j in [1..d], i in [1..4], row in rows_short];
        else
            rows := [Sqrt(2*RealEmbeddings(alpha)[j])*RealEmbeddings(row[i])[j] :
            j in [1..d], i in [1..4], row in rows_short];
        end if;
        matrix := Matrix(RR, 4*d, 4*d, rows);
    end if;
    return matrix;
end intrinsic;





\end{verbatim}

\end{document}